\documentclass[11pt]{amsart}

\usepackage{pgf}
\usepackage{booktabs}
\usepackage{latexsym}
\usepackage{amssymb}
\usepackage{amsfonts}
\usepackage{amsmath}
\newtheorem{theorem}[equation]{Theorem}
\newtheorem{lemma}[equation]{Lemma}
\newtheorem{corollary}[equation]{Corollary}

\def\GL{{\rm GL}}
\def\lcm{{\mathrm lcm}}
\def\F{{\mathbb F}}
\def\Z{{\mathbb Z}}
\def\N{{\mathbb N}}

\def\GL{{\mathrm {GL}}}

\def\SL{{\mathrm {SL}}}
\def\PSL{{\mathrm {PSL}}}

\def\PSU{{\mathrm {PSU}}}
\def\Sp{{\mathrm {Sp}}}
\def\GU{{\mathrm {GU}}}
\def\SU{{\mathrm {SU}}}

\def\GO{{\mathrm {GO}}}

\def\D{{\mathrm {D}}}
\def\GAP{{\sf GAP}}
\def\K{{\mathrm K}}

\pgfdeclareimage[height=4.5cm]{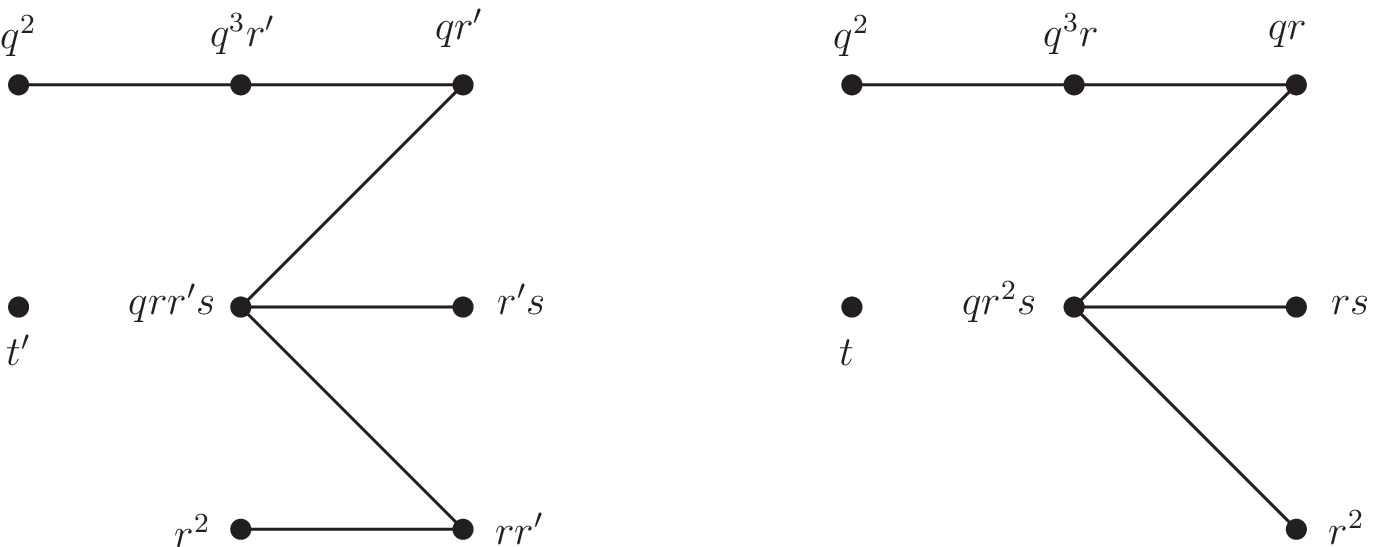}{fig2}

\begin{document}

\title{The Divisibility Graph of  finite groups of Lie Type}

\author[A.Abdolghafourian]{Adeleh~Abdolghafourian}
\author[M.A. Iranmanesh]{Mohammad~A.~Iranmanesh}
\author[A.C. Niemeyer]{Alice~C.~Niemeyer}

\maketitle 

\begin{abstract}
The \emph{Divisibility Graph}
of a finite group  $G$ has vertex set the set of
conjugacy class  lengths of non-central elements in  $G$
and two vertices are connected by an edge if one divides the other.
We determine the connected components of the Divisibility Graph of the
finite groups of Lie type in odd characteristic.
\end{abstract}

\bigskip\noindent
{\small {\sc Keywords:}\quad
Divisibility Graph, Finite Group of  Lie Type
}

\section{Introduction}\label{sec:introd}

Given a set $X$ of  positive integers, several graphs corresponding to
$X$  can  be  defined.   For example  the  \emph{prime  vertex  graph}
$\Gamma(X)$,  the  \emph{common  divisor graph}  $\Delta(X)$  and  the
\emph{bipartite  divisor   graph}  $B(X)$  (see   \cite{IP,Lewis}  and
references   therein  for   more   details).   In   2011  Camina   and
Camina~\cite{CaminaCamina} introduced the  \emph{divisibility graph} $\D(X)$ of
$X$ as the directed graph with  vertex set $X\backslash \{1\}$ with an
edge from  vertex $a$ to vertex  $b$ whenever $a$ divides  $b$.  For a
group $G$  let $cs(G)$ denote  the set  of conjugacy class  lengths of
non-central   elements    in   $G$.     Camina   and    Camina   asked
\cite[Question~7]{CaminaCamina} how  many components  the divisibility
graph  of $cs(G)$  has.   Clearly  it is  sufficient  to consider  the
underlying undirected  graph and for  the remainder of this  paper the
\emph{Divisibility  Graph}  $\D(G)$  of  a group  $G$  refers  to  the
undirected Divisibility Graph  $\D(cs(G))$.

Note that the set of vertices
${cs}(G)$ may be replaced by the set
$\mathcal{C}(G)$ of orders of the centralisers of non-central elements of $G$.

 In \cite{BDIP,Lewis}  has been shown that  the graphs $\Gamma(cs(G)),
 \Delta(cs(G))$ and  $B(cs(G))$ have at most  two connected components
 when $G$ is  a finite group. Indeed, when $G$  is a nonabelian finite
 simple group,  then $\Gamma(cs(G))$ is complete  (see \cite{BHM,FA}).
 It is  clear that $\D(G)$  is a  subgraph of $\Gamma(cs(G)))$  and we
 hope that the structure of $\D(G)$ reveals more about the group $G$.

The first and  second authors have shown that  for every comparability
graph, there is a finite set $X$ such that this graph is isomorphic to
$\D(X)$ in \cite{AIrocky}.  They  found some relationships between the
combinatorial properties of $D(X),  \Gamma(X)$ and $\Delta(X)$ such as
the  number   of  connected   components,  diameter  and   girth  (see
\cite[Lemma~1]{AIrocky}) and   found  a  relationship  between
$\D(X\times Y)$ and product of  $\D(X)$ and $\D(Y)$ in \cite{AIrocky}.
They examined the Divisibility Graph $\D(G)$  of a finite group $G$ in
\cite{AI1} and  showed that when  $G$ is the symmetric  or alternating
group, then  $\D(G)$ has  at most two  or three  connected components,
respectively.  In both cases, at most one connected component is not a
single vertex.

Here we are interested in the  Divisibility Graphs of finite groups of
Lie type.  The graph of  certain finite simple  groups of Lie  type is
known, namely the first and second authors described in \cite{AI2} the
structures   of  the   Divisibility  Graphs   for  $\PSL(2,   q)$  and
$Sz(q)$.  Let $K_i$  denote the  complete graph  on $i$  vertices. They
prove  in  \cite[Theorem~6]{AI2} that  for
$G=\PSL(2,q)$ the graph $\D(G)$ is either $3\K_1$ or $\K_2+2\K_1$.  For
the other finite groups of Lie type in odd characteristic we prove the
following theorem:

\begin{theorem}\label{main}
Let $G$ be a finite group of Lie type over a finite field of order $q$
or $q^2$ in characteristic $p$ where $p$ is an odd prime.  Suppose further that
the Lie rank $\ell$ of $G$ is
either as in Table~$\ref{tab:simple}$ or
at least $2$.   Then the Divisibility Graph  $\D(G)$
has  at most  one  connected  component which is not a
single vertex.
\end{theorem}

In particular,  we prove that  the non-trivial connected  component of
$\D(G)$  contains  the  lengths  of   the  conjugacy  classes  of  all
non-central  involutions  as well  as  the  lengths of  all  conjugacy
classes of all unipotent elements.  The number of connected components
consisting of a  single vertex corresponds to the  number of conjugacy
classes $T^G$  of tori $T$  in $G$  for which $|TZ(G)/Z(G)|$  is odd,
coprime to  $|Z(G)|$ and  for which  the centraliser  in $G$  of every
non-central element in $T$ is $TZ(G).$

We  now compare  the results  of the  theorem to  known results  about
another  type of  graph, namely  the Prime  Graph first  introduced by
Gruenberg and Kegel in 1975  in an unpublished manuscript.  The vertex
set of  the \emph{Prime  Graph} of a  finite group $G$  is the  set of
primes dividing  the order of the  group and two vertices  $r$ and $s$
are adjacent  if and only  if $G$ contains  an element of  order $rs$.
Williams \cite[Lemma~6]{Williams} investigated  Prime Graphs of finite
simple groups in odd characteristic and Kondrat'ev \cite{Kondratev} and
Lucido \cite{lucido} investigated these graphs for even characteristic
and for almost simple groups, respectively.  A subgroup $T$ of a group
$G$ is called a $CC$-group if $C_G(t) \le T$ for all $t\in T\backslash
Z(G)$.  Williams proved  \cite[Theorem~1]{Williams} that the connected
components of the Prime Graph of a finite simple group $G$ of Lie type
in odd  characteristic $p$  consist at  most of  the set  $\{p\}$, the
connected component containing the prime $2$, and a collection of sets
consisting of  the primes  dividing the  order of  some torus  $T$ for
which  $TZ(G)/Z(G)$  has odd  order  coprime  to  $|Z(G)|$ and  is  a
$CC$-group (see Section~\ref{sec:iso}).  Moreover,  for such groups he
showed that  $\{p\}$ is an isolated  vertex of the Prime  Graph if and
only if  $G\cong \PSL(2, q)$  with $q$ odd.   Hence for the  groups we
consider  (see  Section~\ref{sec:ourgroups}  where  we  exclude  small
dimensions and certain `bad' primes) we may assume that $\{p\}$ is not
an isolated vertex  in the Prime Graph.
%We make use of  this fact in
%Lemmas~\ref{lem:isotori}  and~\ref{lem:nonisotori}.
In  this  case,
Williams shows  that the  number of connected  components of  a finite
simple group  of Lie type  is at  most two, unless  $G={}^2D_p(3)$ and
$p=2^n+1$ a  prime for $n\ge2$,  for which  the Prime Graph  has three
components and unless  $G=E_8(q)$, for which the Prime  Graph can have
either  four or  five components,  depending on  whether $q\equiv  2,3
\pmod{5}$ or  $q\equiv 0,1,4 \pmod{5}$, respectively.  We verified the
main theorem separately for the case of small dimensions or the `bad'
primes.  Hence we obtain the following corollary.

\begin{corollary}\label{cor:main}
Let $G$ be a finite group of Lie type over a finite field of order $q$
or $q^2$ in characteristic $p$ where $p$ is an odd prime.  Suppose further that
the Lie rank $\ell$ of $G$ is either as in Table~$\ref{tab:simple}$ or
at  least  $2$ and that $q \not= 3$ if $G$ is of type $E_7$.
Then  the  Divisibility  Graph  $\D(G)$ has  as  many
connected components as the Prime Graph of $G$.
\end{corollary}

We note  that if $G$ is  of adjoint or simply  connected type $E_7(3)$,
the  Divisibility graph of $G$  is  connected,  whereas  the  Prime
Graph  has  three  connected components (see \cite[Table~Id]{Williams}.)

The Prime Graph  of $G$, and thus also the  Divisibility Graph of $G$,
is linked  to yet another  graph defined for $G$,  the \emph{Commuting
  Graph}    introduced    in    1955     by    Brauer    and    Fowler
\cite{BrauerFowler55}.   The vertices of the
Commuting  Graph of  a group  $G$ are the non-central
elements of $G$ and two elements are connected by an edge
if and  only if  they commute. Due  to the work  of Morgan  and Parker
\cite[Theorem~3.7]{MorganParker}   and   Iranmanesh   and   Jafarzadeh
\cite[Lemma~4.1]{IranmaneshJafarzadeh}  we also  obtain the  following
corollary.

\begin{corollary}\label{cor:main2}
Let $G$ be a finite group of Lie type over a finite field of order $q$
or $q^2$ in characteristic $p$ where $p$ is an odd prime.  Suppose further that
$G$ has trivial centre, that the Lie rank $\ell$ of $G$ is either
as in Table~$\ref{tab:simple}$ or
at  least  $2$ and that $p\ge  3$ if $G$ is of type $E_7$.
Then the $G$-classes of connected
components of the commuting graph of $G$
are in one to one correspondence with the connected
components of the  Divisibility  Graph  $\D(G)$.
\end{corollary}

\subsection{Example}

As an example we determine the Divisibility Graph of
the groups $\PSL(3,q)$ and $\PSU(3,q)$ for $q$ odd.
Their conjugacy classes and character tables
can be found in \cite{ss}. In particular, the
centraliser orders of non-trivial elements for both groups are
${\mathcal C}(G)=\{q^3r', q^2, qr'rs, qr', r^2, r'r, r's, t'\}$ where
 $\epsilon=1$ for $G=\PSL(3,q)$ and $\epsilon=-1$ for
$G=\PSU(3, q)$ and $r=q-\epsilon,$ $s=q+\epsilon,$ $t=q^2+\epsilon q+1,$
$a=\gcd(3,r),$  $r'=r/a,$ and $t'=t/a$.
The graph $\D(G)$, shown in Figure \ref{fig:1}, depends on $a$, since when $a=1$
the centraliser orders $r^2$ and $r'r$ agree.
The vertex $t'$ corresponds to a
Coxeter torus in $G$ of odd order which is a $CC$-group.

\begin{figure}[here]
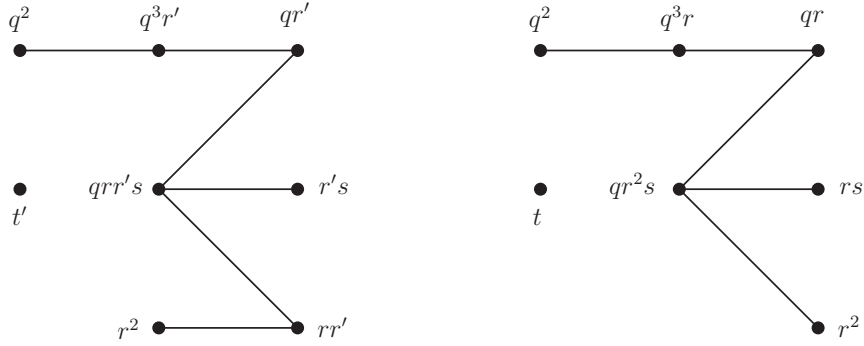

\begin{center}
\pgfuseimage{fig2}
\caption{The divisibility graph for $\PSL(3,q)$ and $\PSU(3,q)$
  (left: $\gcd(3,r)\neq 1$~ right: $\gcd(3,r)=1$).}
\label{fig:1}
\end{center}
\end{figure}

\section{Background}\label{sec:back}

For a finite group $G$ let ${cs}(G) =
\{ |x^G| \mid~x\in G \}\backslash \{1\}$ denote the set of conjugacy class lengths of
non-central elements in $G$.
Let $\D(G)$ denote the Divisibility Graph of $G$,  the graph with vertex set ${cs}(G)$
and edge set ${\mathcal E}(G) =\{ (|x^G|,|y^G|) ;~ either~|x^G|~
divides ~|y^G|~  or
~|y^G|~ divides ~|x^G| \}.$

Note that for $x,y\in G$ there is an edge in $\D(G)$ between $|x^G|$
and $|y^G|$  if and only if
$|C_G (y)|$ divides $|C_G (x)|$ or
$|C_G (x)|$ divides $|C_G (y)|.$
Therefore, the set of vertices
${cs}(G)$ may be replaced by the set
$\mathcal{C}(G)=\{ |C_G(x)| \mid x\ \in G,\, x\not\in Z(G) \}.$
If  $(|x^G|,|y^G|)\in {\mathcal E}(G)$ we write $x\sim y$.
We say non-central elements $x,y\in G$ are \emph{equivalent}
if $|x^G| \not= |y^G|$ and $|x^G|$ and $|y^G|$ are in the same
connected component of $\D(G)$.
 Note that being equivalent induces an equivalence
relation on ${\mathcal C}(G).$

\begin{lemma}\label{lem:sim}
Let $G$ be a finite group and
  $x,y \in G\backslash Z(G)$.
\begin{enumerate}
\item If  $\gcd( |x|,|y| ) = 1$ and $xy = yx$ then
  $C_G(xy) = C_G(x) \cap C_G(y)$ and in particular $x \sim xy \sim y.$
\item $x\sim x^m$ for any $m \in
  \N\backslash \{1\}$ such that $x^m\not\in Z(G).$
\item If $\gcd( |xZ(G)|,|yZ(G)| ) = 1$ and $xyZ(G) = yxZ(G)$ then
$x$ is equivalent to $y.$
\end{enumerate}
\end{lemma}

\begin{proof}
Note that (1) is \cite[Lemma~$3$]{Dolfi}.
(2) is obvious. Now consider (3). By replacing $x$ and $y$
by $x^{m_x}$ and $y^{m_y}$  for some integers $m_x,m_y$ if necessary, we may
assume that $\gcd(|x|,|y|) = 1.$ Now suppose $z\in Z(G)$ such that
$xy=yxz.$ Then $|x|=|x^y| = |xz| = \lcm( |x|,|z|)$ and
$|y|=|y^x| = |yz^{-1}| = \lcm( |y|, |z|).$ Since $|z|$ divides
$\gcd(|x|,|y|)=1$ we have $z=1.$ Hence $xy=yx$ and the result follows from (1).
\end{proof}

\subsection{The groups we consider}\label{sec:ourgroups}
Let $q = p^f$ be a  power of an odd prime $p$ and $f$ a non-negative
integer. Let ${\bf G}$ be a connected semisimple (which implies reductive)
algebraic group of rank $\ell$ with $\ell\ge2$
defined over
the algebraic closure $\overline{\F}_q$ or $\overline{\F}_{q^2}$
of the finite field $\F_q$ or $\F_{q^2}$.
Let $F: {\bf G}\rightarrow {\bf G}$ be a
Frobenius morphism and let  $G = {\bf G}^F = \{g\in {\bf G} \mid F(g)
= g \}$  be a
finite group of Lie type.
Moreover, if ${\bf G}$  is a classical group of Lie type, we
assume that $\Omega \le G \le \Delta$, where $\Omega$ and $\Delta$
and  $\ell$ are as in
Table~\ref{tab:simple}.
For a
field $K$ and a positive integer $n$, for convenience we
denote both $\GL(n,K)$ and $\GU(n,K)$ by $\GL^\epsilon(n,K)$, where
$\epsilon=1$ in the former and $\epsilon=-1$ in the latter case.

Note that  $\SL_2(q) \cong  \Sp_2(q) \cong \SU_2(q)$  so that  we take
$\ell  \ge   3$  for  types  ${}^2A_\ell$   and  $C_\ell$.   Moreover,
$\Omega_3(q) \cong  \PSL_2(q)$ so we  also take  $\ell \ge 2$  in case
$B_\ell$. As $\Omega_4^-(q)  \cong \PSL_2(q^2)$, we take  $\ell \ge 3$
in  case ${}^2D_\ell$.   Moreover, $\Omega_4^+(q)  \cong \SL_2(q)\circ
\SL_2(q)$, where $\circ$ denotes the central product. We determine the
Divisibility Graph of this group directly.  In fact, it is easy to see
that in  this case the  Divisibility Graph is always  connected.  Thus
for the  remainder of  the proof we  may assume $\ell  \ge 3$  in case
$D_\ell$.

\begin{table}[t]
\begin{center}
\begin{tabular}{rcllll}
\quad Type\quad & $n$ & %Simple Group &
$\Omega$ & $\Delta$ & $\dim({\bf \Delta})$ & Rank \\
\toprule
$A_\ell$ &$\ell+1$& %$\PSL(\ell+1,q)$  &
$\SL(n,q)$& $\GL(n,q)$ &
$n^2$& $\ell\geq 3$\\
${}^2A_\ell$ & $\ell+1$ &%$\PSU(\ell+1,q)$&
$\SU(n,q)$& $\GU(n,q)$& $n^2$  &$\ell\geq 3$\\
$B_\ell$ & $2\ell+1$ & %$\P\Omega(2\ell+1,q)$ &
$\Omega(n,q)$&
$\GO(n,q)$&  $n(n-1)/2$ &$\ell\geq 2$\\
$C_\ell$ & $2\ell$ &  %$\PSp(2\ell,q)$ &
$\Sp(n,q)$& $\mbox{GSp}(n,q)$& $n(n+1)/2$
 &$\ell\geq 3$\\
$D_\ell$ & $2\ell$ &%$\P\Omega^{+}(2\ell,q)$  &
 $\Omega^{+}(n,q)$&
$\GO^{+}(n,q)$&  $n(n-1)/2$ &
$\ell\geq 2$\\
${}^2D_\ell$ & $2\ell$ & %$\P\Omega^{-}(2\ell,q)$  &
$\Omega^{-}(n,q)$&
$\GO^{-}(n,q)$&  $n(n-1)/2$ &
$\ell\geq 3$\\
\bottomrule
\end{tabular}\\
\medskip
\caption{Finite classical groups considered in Theorem~\ref{main}}\label{tab:simple}
\end{center}
\end{table}

For convenience we also record the dimension of the algebraic group
corresponding to $\Delta$
in the same table.

For  the exceptional  groups  when $p$  is  a bad  and  odd prime,  we
obtained the Divisibility graphs directly,  either using the Tables in
\cite{Der1983,Eno,FJ93,FJ94,Shoji} or  using an  explicit list  of the
generic centraliser orders in $E_7(q)$  as polynomials in $q$ computed
by Frank L\"ubeck.  Using L\"ubeck's description as polynomials in $q$
in {\sf GAP}  \cite{GAP}, we could first determine the  lengths of the
conjugacy  classes of  $E_7(q)$ as  polynomials in  $q$ together  with
their  multiplicities.   For  those conjugacy  classes  with  non-zero
multiplicities,  we  could  determine   the  divisibility  graph.   We
verified that  Theorem~\ref{main} holds  for all  such groups.  In the
case  of  the  adjoint group  $E_7(q)$  for  $q$  a  power of  3,  the
divisibility graph is connected even when $q=3.$

Therefore, from now on we assume that $p$ is odd and a good prime for
${\bf G}$, that is
(\cite[p. 28]{Carter})
\begin{itemize}
\item $p \not= 2$ when $G$ has type $A_\ell,{}^2A_\ell,
  ~B_\ell,~C_\ell,~D_\ell, {}^2D_\ell$,\\
\item $p \not\in \{2,3\}$ when $G$ has type $G_2, F_4, E_6, {}^2E_6, E_7$,\\
\item $p \not\in \{2,3,5\}$ when $G$ has type $E_8$.
\end{itemize}
The force of $p$ being a good prime is that for any $x\in {\bf G}$ all
unipotent elements in $C_{\bf G}(x)$ lie in $C_{\bf G}(x)^\circ.$

Let $F: {\bf G}\rightarrow {\bf G}$ be a
Frobenius morphism and let  $G = {\bf G}^F = \{g\in {\bf G} \mid F(g)
= g \}$  be a
finite group of Lie type, ${\bf T}_0$  a fixed  $F$-stable maximal torus and $W =  N_{\bf G}({\bf T}_0)/{\bf T}_0$ the Weyl group of ${\bf G}.$
Throughout the paper, let $\overline{G}$ denote $G/Z(G)$ and $\overline{g} =
gZ(G)$ for $g\in G$.

An element $g\in G$ is called \emph{regular}, if $C_{\bf G}(g)$ has
dimension $\ell$, where $\ell$ denotes the rank of ${\bf G}$, i.e.\ the
dimension of a maximal torus of ${\bf G}.$
More background on the finite groups of Lie type can be found in
\cite{Carter} and \cite{MalleTesterman}.

\section{Proof of the main theorem}

Our aim is to determine the connected components of the Divisibility
Graph $\D(G)$. It is well known that each element $g\in G$ has a
Jordan decomposition $g=us=su$ with $u$ unipotent and $s$
semisimple. If $g$ is non-central then either  $s$ is non-central and
hence $g\sim s$ or $u$ is non-trivial and $g\sim u.$ Thus we consider
non-central semisimple and unipotent elements.
In Section~\ref{sec:uni} we show that all unipotent elements in $G$ are
equivalent.
In Section~\ref{sec:semi} we show that every non-central and
non-regular semisimple element is equivalent to a
unipotent element in $G$.
In Section~\ref{sec:iso} we recall the notion of a $CC$-group
introduced by Williams and show that certain maximal tori in $G$ which are
also $CC$-groups yield isolated vertices of $\D(G)$.
Finally we show in Section~\ref{sec:regsemi} that all regular
semisimple elements which do not lie in such a  torus
are also equivalent to a unipotent element. Thus the main theorem follows.

\subsection{Unipotent elements}\label{sec:uni}

Let $G$ be a finite classical group of Lie type and let $u$ be a
unipotent element in $G$.
We determine the $q$-part of the order of the centralizer of $u$ in
$\Delta$ when $G$ is one of the groups in Table~\ref{tab:simple}. As $q$ does
not divide $|Z(G)|$ nor $[\Delta:\Omega]$, the $q$-part
of $|C_H(u)|$ is the same for all groups $H$ with $\Omega \le H
\le\Delta$  and for groups $H/Z(H).$

\begin{lemma}\label{lem:qpart}
Let ${\bf G} = \GL_n(\overline{\F}_q), Sp_n(\overline{\F}_q)$ or
$GO_n(\overline{\F}_q)$ and let $G={\bf G}^F$.
Let $u$ be a unipotent element in $G$ which corresponds in ${\bf G}$ to
a  Jordan decomposition $\oplus J_i^{r_i}$. In particular, in cases
$\Sp$ or $O$ the integer $r_i$ is even for $i$ odd or $i$ even, respectively.
Then $|C_{G}(u)|_q = q^{a_u}$, where $a_u$ is determined by
  the following formulas:
\begin{enumerate}
\item $G =\GL^\epsilon(n,q)$ then
  ${\displaystyle a_u = \sum_{i} \left( ir_i^2 -\frac{r_i(r_i+1)}{2} \right)
  + 2\sum_{i<j} ir_ir_j}.$
\item $G =\Sp(n,q)$ then
  ${\displaystyle a_u =
  \frac{1}{2}\sum_{i} (i-\frac12)r_i^2
  + \sum_{i<j} ir_ir_j +  \sum_{i\mbox{\ \tiny even}\atop r_i \mbox{\ \tiny odd}}\frac14}.$
\item $G =\GO^\epsilon(n,q)$ for $\epsilon \in \{\circ,+,-\},$ then \\
  ${\displaystyle a_u =
  \frac{1}{2}\sum_{i} (i-\frac12)r_i^2
  + \sum_{i<j} ir_ir_j    -\frac12 \sum_{i} r_i
+ \sum_{i, r_i\mbox{ \tiny odd}}\frac{1}{4}}.$
\end{enumerate}
\end{lemma}

\begin{proof}
We first note that it is well known that
 $C_{\bf G}(u)$ admits a Levi-decomposition, see for example
\cite[Prop.~3.2]{GoodwinRoehrle} or \cite[Theorem~3.1]{LS}. That is,
 $C_{\bf G}(u) = {\bf UR}$, where ${\bf U} = R_u(C_{\bf G}(u))$ is the
unipotent radical of $C_{\bf G}(u)$, the
group ${\bf R}$ is reductive
and ${\bf U}\cap {\bf R} = \{1\}.$
Moreover, ${\bf U}$ and ${\bf R}$ are $F$-stable,
$C_{G}(u) = UR$,   and  $|U| = q^{\dim {\bf U}},$
 see \cite[Prop.~3.2]{GoodwinRoehrle} or
 \cite[Theorem~7.1]{LS}.
Thus we need to
determine ${\dim\bf U}.$ Consequently,
$\dim {\bf U} = \dim C_{\bf G}(u) -\dim {\bf R}.$ We consider the different
cases of classical groups in turn and note that in each case
\cite[Theorem~3.1(iii)]{LS} determines $\dim C_{\bf G}(u)$ and
\cite[Theorem~3.1(iv),Theorem~7.1(ii)]{LS}  describe
${\bf R}$ and thus $ \dim {\bf R}$, and $R$.  It follows that
$|C_{G}(u)|_q = q^{\dim {\bf U}} |R|_q$.

{\bf Case GL}:\quad Let $G= \GL^\epsilon(n,q)$.
Then
$\dim C_{\GL^\epsilon(n,\overline{\F}_q)}(u) = \sum_{i} ir_i^2
  + 2\sum_{i<j} ir_ir_j$ and
${\bf R} = \prod \GL(r_i,\overline{\F}_q),$
thus $\dim {\bf R} = \sum_i
r_i^2$, and
$R = \prod \GL^\epsilon(r_i,\overline{\F}_q),$
 showing that
$\dim {\bf U} =  \sum_{i} (i-1)r_i^2 + 2\sum_{i<j} ir_ir_j.$
Thus $|C_{\GL(n,q)}(u)|_q = q^{a_u} = q^{\dim({\bf U}) }\prod_i
  |\GL(r_i,q)|_q.$ Thus
\begin{eqnarray*}
a_u
&=& \sum_{i} (i-1)r_i^2 + 2\sum_{i<j} ir_ir_j + \sum_i \frac{r_i(r_i-1)}{2}\\
&=&  \sum_{i} \left( ir_i^2 -\frac{r_i(r_i+1)}{2} \right)
  + 2\sum_{i<j} ir_ir_j.
\end{eqnarray*}

{\bf Case Sp}:\quad
Here
${\displaystyle \dim C_{\bf G}(u) = \frac{1}{2}\sum_{i} ir_i^2
  + \sum_{i<j} ir_ir_j +  \frac12\sum_{i \mbox{ \tiny odd}} r_i}$
and, moreover,
${\displaystyle {\bf R} =
\prod_{i \mbox{ \tiny odd}} \Sp(r_i,\overline{\F}_q)
\times
\prod_{i \mbox{ \tiny even}} GO(r_i,\overline{\F}_q)}.$
Thus, $\dim {\bf R} =
\frac12 \sum_{i \mbox{ \tiny odd}}
r_i(r_i+1) +\frac12 \sum_{i \mbox{ \tiny even}} r_i(r_i-1)$,
 showing that
\begin{eqnarray*}
\dim {\bf U} &=&  \dim C_{\bf G}(u) - \dim {\bf R} \\
&=&
\frac{1}{2}\sum_{i} ir_i^2
  + \sum_{i<j} ir_ir_j +  \frac12\sum_{i \mbox{ \tiny odd}} r_i -
\frac12 \sum_{i \mbox{ \tiny odd}}
r_i(r_i+1) -\frac12 \sum_{i \mbox{ \tiny even}} r_i(r_i-1)\\
&=&
\frac{1}{2}\sum_{i} ir_i^2
  + \sum_{i<j} ir_ir_j - \frac12 \sum_{i \mbox{ \tiny odd}}
r_i^2 -\frac12 \sum_{i \mbox{ \tiny even}} r_i(r_i-1)\\
&=&
\frac{1}{2}\sum_{i} (i-1)r_i^2
  + \sum_{i<j} ir_ir_j
   +\frac12 \sum_{i \mbox{ \tiny even}} r_i.
\end{eqnarray*}
It follows that
$q^{a_u} = q^{\dim({\bf U})} \prod_{i \mbox{ \tiny odd}}|\Sp(r_i,q)|_q
\prod_{i \mbox{ \tiny even}} |GO^{\epsilon_i}(r_i,q)|_q$ and
\begin{eqnarray*}
a_u 
&=&  \frac{1}{2}\sum_{i} (i-1)r_i^2
  + \sum_{i<j} ir_ir_j
 + \sum_{i}\frac{r_i^2}{4}
+\sum_{{i\mbox{ \tiny even}}\atop{ r_i\mbox{\tiny odd}}}\frac14. \\
\end{eqnarray*}

{\bf Case O}:\quad
Here
${\displaystyle \dim C_{\bf G}(u) = \frac{1}{2}\sum_{i} ir_i^2
  + \sum_{i<j} ir_ir_j -  \frac12\sum_{i \mbox{ \tiny odd}} r_i}$
and furthermore
${\displaystyle {\bf R} = \prod_{i \mbox{ \tiny odd}} GO(r_i,\overline{\F}_q)
\times
\prod_{i \mbox{ \tiny even}} \Sp(r_i,\overline{\F}_q)}.$
Thus $\dim {\bf R} =
 \sum_{i \mbox{ \tiny odd}}
\frac{r_i(r_i-1)}{2} + \sum_{i \mbox{ \tiny even}} \frac{r_i(r_i+1)}{2}$,
 showing that
\begin{eqnarray*}
\dim {\bf U} &=&  \dim C_{\bf G}(u) - \dim {\bf R} \\
&=&
\frac{1}{2}\sum_{i} ir_i^2
  + \sum_{i<j} ir_ir_j -  \frac12\sum_{i \mbox{ \tiny odd}} r_i -
\frac12 \sum_{i \mbox{ \tiny odd}}
r_i(r_i-1) -\frac12 \sum_{i \mbox{ \tiny even}} r_i(r_i+1)\\
&=&
\frac{1}{2}\sum_{i} ir_i^2
  + \sum_{i<j} ir_ir_j - \frac12 \sum_{i \mbox{ \tiny odd}}
r_i^2 -\frac12 \sum_{i \mbox{ \tiny even}} r_i(r_i+1)\\
&=&
\frac{1}{2}\sum_{i} (i-1)r_i^2
  + \sum_{i<j} ir_ir_j    -\frac12 \sum_{i \mbox{ \tiny even}} r_i.
\end{eqnarray*}
Therefore,
$q^{a_u} =q^{ \dim({\bf U})} \prod_{i \mbox{ \tiny odd}}|GO^{\epsilon_i}(r_i,q)|_q
\prod_{i \mbox{ \tiny even}} |\Sp(r_i,q)|_q$ and
\begin{eqnarray*}
a_u
&=&  \frac{1}{2}\sum_{i} (i-\frac12)r_i^2
  + \sum_{i<j} ir_ir_j    -\frac12 \sum_{i} r_i
+ \sum_{i, r_i\mbox{ \tiny odd}}\frac{1}{4}. \\
\end{eqnarray*}

\end{proof}

We now show that all unipotent elements in $G$ are equivalent.

\begin{lemma}\label{lem:uni}
Let $G$  be a  finite group  of Lie type  as  given in
Section~$\ref{sec:ourgroups}$.
Then all unipotent elements in $G$ are equivalent.
\end{lemma}

\begin{proof}
Let  $u$  be a  regular  unipotent element  in  $G$,  which exists  by
\cite[Prop.5.1.7]{Carter}.
Then $u$ lies in a maximal
connected $F$-stable unipotent subgroup ${\bf U}$ of ${\bf G}$ and  by
\cite[III.1.14]{SpringerSteinberg}
$C_{\bf G}(u) = Z({\bf G}).C_{\bf  U}(u)$ and, since
$p$ is good, $C_{\bf U}(u)$ is connected.
Moreover, $|C_{G}(u)|_q  = q^{\ell}$.

Note that it is enough to show this for one of the groups with
given Dynkin Diagram as their $q$-parts of the centralisers are the
same. For the classical groups we restrict our attention to the
groups considered in Lemma~\ref{lem:qpart}.

Let $v$ be unipotent with  Jordan decomposition $\oplus_i
J_i^{r_i}$.  By Lemma~\ref{lem:qpart} $|C_{G}(v)|_q  = q^{a_v}.$
If $v$ is not
regular, we show   that $q^\ell$ divides $|C_{G}(v)|_q. $ From this it follows
 that $v$ is equivalent to $u$, as
all semisimple elements in $C_G(u)$ lie in $Z(G)$ and $Z(G) \le
C_G(v).$
Thus we need to show that  $a_v - \ell \ge 0.$  We
note that $n = \sum_{i} ir_i. $

Suppose first that $G = \GL^\epsilon(n,q)$. By Lemma~\ref{lem:qpart}
$|C_{\GL(n,q)}(v)|_q  = q^{a_v}$ with
$a_v= \sum_{i} \left( ir_i^2 -r_i(r_i+1)/2 \right)
  + 2\sum_{i<j} ir_ir_j$.
As $ir_ir_j \ge r_ir_j,$ we note that \[a_v -\ell \ge
\sum_{i} r_i\left( ir_i -i -\frac{(r_i+1)}{2} + \sum_{j\not=i} r_j \right).\]
We show that for
any $i$  the $i$-th summand is non-negative. Clearly this is the case
when $r_i=0.$  So suppose now $r_i > 0.$ Then it suffices to show
$(i-\frac12)r_i -(i+\frac12) + \sum_{j\not=i} r_j
 \ge 0$ if
$ r_i \ge \frac{i+\frac12}{i-\frac12} = 1 + \frac1{i-\frac12}$.
For $i = 1$ this is the case for
$r_i \ge 3$ and for $i > 1$ this is the case for  $r_i \ge 2.$
Hence for any such pair $(i,r_i)$ each summand is non-negative.

Now consider  the case $i=1$ and $r_i=1,2$ or
 $i > 1$ and $r_i=1$. As $v$ is not regular, $i < n.$
And since  $n \ge 3$ in either case, there is at least one $j$ with
$j\not=i$ such that  $r_j \not=0$. Then
$(i-\frac12)r_i -(i+\frac12) + \sum_{j\not=i} r_j
 \ge  (i-\frac12)r_i - i+\frac12.$ For $r_i=1$ this expression is $0$ and
for $r_i=2$ we have $i=1$ and thus
$(i-\frac12)r_i - i+\frac12 = 1/2 \ge 0.$

Suppose now that $G = \Sp(n,q)$ with $n$ even.
As $ir_ir_j \ge r_ir_j,$  and  $\ell = \frac12 n =
 \frac12  \sum_{i} ir_i$, we note that by
 Lemma~\ref{lem:qpart}(b)
$|C_{\Sp(n,q)}(v)|_q  = q^{a_v}$ with
 $a_v -\ell \ge
 \frac{1}{2}\sum_{i} r_i\left( (i-\frac12)r_i +\sum_{j\not=i}r_j
 -i\right).$

It suffices to show that for
any $i$  the $i$-th summand is non-negative. This is certainly the
case when $r_i=0.$ So suppose $r_i>0.$
This is the case when  $r_i\ge 2.$
Now suppose $r_i=1$.
 Then $i$ has to be even and hence $i\ge 2.$
As  $v$ is not regular and
$v\not=1$, there is a $j$ with $r_j \not=0$, whence
$\sum_{j\not=i}r_j \ge 1$,  and now
$(i-\frac12)r_i + 1 -i \ge 0.$

Suppose now that $G = \GO^\varepsilon(n,q)$ with $\epsilon = \circ$, if $n$ odd and
$\epsilon = \pm$  if $n$ even.
The Lie rank $\ell$ of $G$ is $\frac12 (n-1)$ when $n$ is
odd and $\frac12 n$ when $n$ is even.  In either case,
$ \frac12  \sum_{i} ir_i \ge \ell.$

As $ir_ir_j \ge r_ir_j,$ we note that
for  $|C_{\GO^\varepsilon(n,q)}(v)|_q  = q^{a_v}$  we have
by Lemma~\ref{lem:qpart}(3)
\begin{eqnarray*}
a_v - \ell &\ge&
\frac{1}{2}\sum_{i} r_i\left( (i-\frac12)r_i -1 -i+\sum_{j\not=i}r_j
\right)
+ \frac{1}{2}\sum_{i, r_i\mbox{ \tiny odd}}\frac{1}{2}
\end{eqnarray*}
It suffices to show that for
any $i$  the $i$-th summand is non-negative. This is certainly the
case when $r_i=0.$ So suppose $r_i>0.$
If $r_i\ge 2$, then as $\sum kr_k \ge 4$, either $i\ge 2$
or there is a $j$ with $r_j \not=0$, whence
$\sum_{j\not=i}r_j \ge 1$.  Thus the $i$-th summand is non-negative as
 either
$ (i-\frac12)2 -1 -i = i - 2 \ge 0$
or
$ (i-\frac12)2 -1 -i+1 = i-1 \ge 0.$
Now suppose $r_i=1$. Then $i$ has to be odd and,
as  $v$ is not regular and
$v\not=1$, there is a $j$ with $r_j \not=0$, whence
$\sum_{j\not=i}r_j \ge 1$.
Thus the $i$-th summand of $a_v-\ell$ is  at least
$\frac{1}{2} (i-\frac12 -1 -i+ 1)  + \frac{1}{4}   = 0.$

We now consider the remaining cases. Note that it suffices to consider
the simple group.

{\bf $G_2(q)$}: \quad Since $p$ is good, $\gcd(6,p) = 1$ and $G$
contains  a unipotent element $u$ whose centraliser has order $q^2$
and $q^2$ divides the order of the centraliser of every other
non-trivial unipotent
element in $G$ by \cite[Table~22.2.6]{LS}.

{\bf $F_4(q)$}: \quad Since $p$ is good, $\gcd(12,p^2) = 1$ and $G$
contains  a unipotent element $u$ whose centraliser has order $q^4$
and $q^4$ divides the order of the centraliser of every other
non-trivial unipotent
element in $G$ by \cite[Table~22.2.4]{LS}.

{\bf $E_6(q), {}^2E_6(q)$}: \quad Since $p$ is good, $\gcd(6,p) = 1$ and $G$
contains  a unipotent element $u$ whose centraliser has order $q^6$
and $q^6$ divides the order of the centraliser of every other
non-trivial unipotent
element in $G$ by \cite[Table~22.2.3]{LS}.

{\bf $E_7(q)$}: \quad Since $p$ is good, $\gcd(12,p^2) = 1$ and $G$
contains  a unipotent element $u$ whose centraliser has order $q^7$
and $q^7$  divides the order of the centraliser of every other
non-trivial unipotent element in $G$ by \cite[Table~22.2.2]{LS}.

{\bf $E_8(q)$}: \quad Since $p$ is good, $\gcd(60,p^2) = 1$ and $G$
contains  a unipotent element $u$ whose centraliser has order $q^8$
and $q^8$ divides the order of the centraliser of every other
non-trivial unipotent element in $G$ by \cite[Table~22.2.1]{LS}.

\end{proof}

 \subsection{Non-regular semisimple elements}\label{sec:semi}

 \begin{lemma}\label{lem:semi-non-reg}
Let $G$ be as in Section~$\ref{sec:ourgroups}$.
 Then every non-central, non-regular semisimple element
 is equivalent to a unipotent element of $G$.
 \end{lemma}

 \begin{proof}
 Let $s$ be a non-central and  non-regular semisimple element in $G$. Then
 $s$ lies in  an $F$-stable torus ${\bf T}$ of ${\bf G}$  and
 $C_{\bf G}(s)^\circ = \langle {\bf T}, X_\alpha \mid \alpha(s) = 1 \rangle,$
 where the  $X_\alpha$ denote the root subgroups with respect to the
 torus ${\bf T}$ by \cite[Theorem 3.5.3(i)]{Carter}.
 Now $C_{\bf G}(s)^\circ$ is a connected reductive and $F$-stable
 \cite[p. 28]{Carter} and $C_{\bf G}(s)^F = C_G(s)$ (see \cite[p. 1]{CarterPLMS}).
 Moreover,
 since $s$ is not regular,
 there exists a root $\alpha$ with respect to
 ${\bf T}$ for which $\alpha(s) = 1$ by \cite[Proposition
   14.6]{DigneMichel}.
 Therefore the equivalence class $A$ of $X_\alpha$ determines a
 non-trivial unipotent  subgroup $(X_A)^F$ of $G$  by \cite[Prop.~23.7,
   Prop.~23.9]{MalleTesterman}. In particular
  $C_{{G}}(s)$ contains a unipotent element.
 \end{proof}

Our next aim  is to show that if $xZ(G)$ is  an involution in $G/Z(G)$
then  $x$  is equivalent  to  a unipotent  element  in  $G$.  For  the
exceptional groups the proof relies  on the knowledge of their classes
of involutions.

\begin{lemma}\label{lem:inv}
Let $G$  be a  finite group  of Lie type  of rank  $\ell$ as  given in
Section~$\ref{sec:ourgroups}$.   If   $\overline{x}$  is  an   involution  in
$\overline{G}$ then $x$ is equivalent to a unipotent element in $G$.
\end{lemma}

\begin{proof}
Suppose  first  that $G$  is  a  group  of type  $A_\ell,  {}^2A_\ell,
B_\ell,$  $C_\ell$.  As  $q$  odd, an element of order a power of two
is  semisimple.  Note also that a non-central semisimple element $x \in
G$ is regular
if and only if it lies in a unique maximal torus of $G$. This is the
case if and only if $xZ(G)$ is regular in $G/Z(G)$.
A  semisimple element in a classical group  $H$  in natural
characteristic $p$ is
regular in types $A_\ell, {}^2A_\ell, B_\ell, C_\ell$ if and only if
its minimal polynomial and its characteristic polynomial are
equal. In case $D_\ell$ a stronger condition holds,
see \cite[Theorem~3.2.1]{FNP}.
For all values of $\ell$ we consider (see Section~\ref{sec:ourgroups})
there are no regular
involutions in classical groups. Moreover, there are also no regular
elements of order four
whose square is a central involution. Thus in $H$ or $H/Z(H)$
every involution is equivalent to a unipotent element by
Lemma~\ref{lem:semi-non-reg}.

Now let $\overline{G}$ be one of the exceptional or twisted simple
groups, namely $G_2(q),$ $F_4(q),$ $E_6(q),E_7(q),$
$E_8(q),{}^2E_6(q),{}^3D_4(q)$ where $q$ is a power of an odd prime.
Character tables and conjugacy classes  of these finite simple groups of Lie
type  have been  studied  extensively, see  for example
\cite{Chang,Der1983,DM,ModularAtlas,Mizuno,Shoji}
and it is known that the order of a centraliser of an involution in
these groups is divisible by $q.$
So by Lemma \ref{lem:sim}, each involution is equivalent to
some unipotent element.
\end{proof}

\begin{lemma}\label{lem:evencent}
Let $G$ be as in Section~$\ref{sec:ourgroups}$ and
let $s$ be a  semisimple element  of $G$.
If $|C_{\overline{G}}(\overline{s})|$ is even, $s$ is equivalent to a
unipotent element.
\end{lemma}

\begin{proof}
As $|C_{\overline{G}}(\overline{s})|$ is even,  there
is an involution  $\overline{x}\in  C_{\overline{G}}(\overline{s}).$

If $|\overline{s}|$ is odd, then some power of $\overline{s}^b$ has
odd prime order
and thus $\overline{s}\sim \overline{s}^b\sim \overline{x}$ as $\overline{x}\overline{s} =
\overline{s}\overline{x}$ and $\gcd(|\overline{s}|,|\overline{x}|) = 1$ by Lemma~\ref{lem:sim}(1).
By Lemma~\ref{lem:sim}(3) $s^b$  and thus also $s$ is equivalent to
$x$.
If $|\overline{s}|$ is even then we may  choose $x$ to be the power
$s^b$ for which $\overline{s}^b$ is an involution in $\overline{T}_0.$
Clearly $s$ is equivalent to $x=s^b$.
By Lemma~\ref{lem:inv} $x$ is equivalent in $G$ to a unipotent element
and thus so is $s$.
\end{proof}

\subsection{Isolated vertices}\label{sec:iso}

Williams  \cite[Lemma~5]{Williams} calls  a subgroup  $T$ of  a finite
group  $G$  a  \emph{CC-group}  if   $C_G(x)  \le  T$  for  all  $x\in
T\backslash  Z(G).$  If  $G$  is  a  group
we denote by $\overline{G}$ the group
$G/Z(G)$, for $H\le G$ we let $\overline{H}=HZ(G)/Z(G)$
 and for $g\in G$ we denote $gZ(G)$ by $\overline{g}.$
Now suppose $G$ is a finite group  of  Lie  type
$\overline{T}$    is    a    maximal    torus    of
$\overline{G}$      and     a     $CC$-group      such     that
$\gcd(|\overline{T}|,|Z(G)|)  = 1$. Williams  proved that  the set  $\pi$ of
prime divisors of $|\overline{T}|$  forms a connected component of the
Prime Graph of  $G$. Some general properties of  $CC$-groups are given
in   \cite[Proposition  1.14]{Babaietal},   where   they  are   called
\emph{sharp} subgroups. In particular,
a torus which
is a $CC$-group is a Hall $\pi$-subgroup of $G$.

Here we prove that if $G$ is a finite group of Lie type and
$T$ is a maximal torus such that $\overline{T}$ has odd order and
is a  $CC$-group in
$\overline{G}$, then $|s^G|$ for an $s\in T\backslash Z(G)$
forms an isolated vertex of $\D(G).$
Let $T^G=\cup_{h\in G}T^h$.

If the Prime Graph of $G$ has more than one component, let $\pi_1$
denote the component containing the prime 2. It follows from
\cite[p.~487]{Williams}
the number of components of the Prime Graph of
$G$ is at most that of the Prime Graph of the simple factor
corresponding
to $G$. In the following two lemmas we make use of a result of Williams
\cite[Lemma~5]{Williams} in which he showed in particular that for a
(not necessarily simple)
group $G$ as in   Section~$\ref{sec:ourgroups}$ and a maximal torus
$T$ in $G$ the primes dividing $|T|$ form a complete component of the
Prime Graph of $G$ not containing the prime 2 if and only if
$|\overline{T}|$ is odd, coprime to $|Z(G)|$ and $\overline{T}$
is a $CC$-group.

\begin{lemma}\label{lem:isotori}
Let $G$ be as in Section~$\ref{sec:ourgroups}$.
Let $T$ be a maximal torus
in $G$  for which $\overline{T}$ has odd order coprime to $|Z(G)|$ and
is a $CC$-group.
Then the non-central elements $g$ of $T^G$ form the isolated vertex
$|g^G|=|G|/|T Z(G)|$ in $\D(G).$
\end{lemma}

\begin{proof}
Let $a=g^h\in T^G\backslash Z(G).$
Then $|C_G(a)|=|C_G(g^h)| = |C_G(g)|=|T Z(G)|$ for every $a\in
T^G\backslash Z(G)$  and thus
 the non-central elements in $T^G$ form a single vertex in $\D(G)$.
It remains to see that this vertex is isolated.

Suppose to the contrary that $g\in T^G$ is non-central and its conjugacy
class length is not isolated
in $\D(G)$. Then there is a non-central element $x\in G\backslash T^G$ with
$x \sim g$ but $|x^G| \not= |g^G|.$
Let $\pi$ denote the set of prime divisors of $|\overline{T}|$.
Note that by \cite[Lemma~5]{Williams} $\pi$ is a complete component of
the Prime Graph of $G$ containing only odd primes.
As $x\not\in T^G$, the elements $\overline{g}$ and $\overline{x}$ have
coprime order.

Let $b$ be  a prime  with $b\mid
  |\overline{x}|$ and $b\not\in \pi.$ Thus $b$ divides $|C_G(x)|$ but not
$|C_G(g)|$, hence $|C_G(g)| \mid |C_G(x)|$ since $x\sim g.$
Let $r\in \pi$ and let $y\in C_G(x)$ be an
  element of order $r$. Then $x\in C_G(y)$.
As $\gcd( |\overline{T}|, |Z(G)|) = 1$ it follows
  that $y \not\in Z(G)$ and hence $\overline{x}\in
  C_{\overline{G}}(\overline{y}).$ Therefore
  $\overline{x}\overline{y}$ is an element of order $br$ in
$ C_{\overline{G}}(\overline{y})$ and $r$ is connected to $b$ in the
  Prime Graph of $\overline{G}$. This is a contradiction to $
\pi$ being a complete connected component of the Prime Graph.
\end{proof}

\begin{lemma}\label{lem:nonisotori}
Let $G$ be as in Section~$\ref{sec:ourgroups}$.
Let $T$ be a maximal torus
in $G$  for which $\overline{T}$ is a $CC$-group and
$\gcd(|\overline{T}|, |Z(G)|)\not=1.$ Then $|\overline{T}|$ is even.
\end{lemma}

\begin{proof}
Let $\pi$ denote the set of divisors of $|\overline{T}|.$ If
$|\overline{T}|$ is even, the result holds. Seeking a contradiction,
we now assume   $|\overline{T}|$ is odd. As
$\gcd(|\overline{T}|, |Z(G)|)\not=1$, there
 is an odd prime $r\in \pi$ dividing $|Z(G)|$. In particular, we are then in case
$A_\ell, {}^2A_\ell$, $E_6(q)$ or ${}^2E_6(q).$ In all of these cases,
$r$ divides the centraliser of an involution in $\overline{G}$,
see the proof of
\cite[Lemma~5(d)]{Williams}. In particular,
 there is an  element $g\in G$ for which $\overline{g}$ has order
 $2r.$ Let $x = g^r$   and $y = g^2$. Since $\overline{T}$ is a
 $CC$-group it is a Hall $\pi$-subgroup there is an
element $\overline{h}\in \overline{G}$ such that
$\overline{t} = \overline{y}^{\overline{h}} \in \overline{T}$. In particular,
$\overline{t}$ has order $a$ and commutes with
$\overline{x}^{\overline{h}}$ which has order $2$. Thus  $2$ divides
$|C_{\overline G}(\overline{t})| = |\overline{T}|$, a contradiction to
our assumption.
\end{proof}

\subsection{Regular semi-simple elements}\label{sec:regsemi}

We now consider regular semisimple elements. For such an element $s$
we know that $C_{\bf G}(s)^\circ = {\bf T}$ (see
\cite[Cor.~14.10]{MalleTesterman}) and hence
${(C_{\bf G}(s)^\circ)}^F = {\bf T}^F = T.$

We first  identify those regular semisimple elements in a finite group
$G$ of Lie type whose centralizers are equal to the maximal torus $T$.
Let $k$ denote the order of the fundamental group of $G$.
Table~\ref{tab:k} yields $k$ (see \cite[pp. 25-26]{Carter}).

\begin{table}[t]
\begin{center}
\begin{tabular}{ll}
Dynkin Diagram & $k$\\
\hline
$A_\ell$ &  $\gcd(\ell+1,q-1)$\\
${}^2A_\ell$ &  $\gcd(\ell+1,q+1)$\\
$B_\ell, C_\ell$ & 2 \\
$D_\ell$ & 4 \\
$G_2, F_4, E_8$ & 1 \\
$E_6$ & 3 \\
$E_7$ & 2 \\
\end{tabular}
\caption{Orders of the fundamental group}\label{tab:k}
\end{center}
\end{table}

The following lemma is \cite[Lemma~4.4]{SpringerSteinberg},
\cite[Proposition~14.20]{MalleTesterman} and \cite[p.~25]{Carter}.
For a semisimple element we use the notation
$C_{\overline{G}}(\overline{s}) = C_{\overline {\bf    G}}(\overline{s})^F$   and
$C_{\overline { G}}(\overline{s})^\circ =  \left(C_{\overline {\bf   G}}(\overline{s})^\circ\right)^F.$

\begin{lemma}\label{lem:centconnected}
Let $G$ be as in Section~$\ref{sec:ourgroups}$ and
let $s$ be a  semisimple element in $G.$
Then $[C_{\overline G}(\overline{s}):C_{\overline
    G}(\overline{s})^\circ]$  divides  the order $k$ of the
Fundamental group  as in Table~\ref{tab:k}.
If $\gcd(|s|, k)=1$
then  $C_{G}(s) = C_{ G}(s)^\circ$.   In particular,
if $\gcd(|s|, k)=1$  and $s$ is regular semisimple, then the unique
maximal torus $T$
containing $s$ is $C_G(s).$
\end{lemma}

 \begin{lemma}\label{lem:aux}
Let $G$ be as in Section~$\ref{sec:ourgroups}$ and
let $s$ be a regular semisimple element in $G.$
 If $C_{\overline G}({\overline s}) > C_{\overline
   G}(\overline{s})^\circ=\overline{T}$ then $s$ is
 equivalent to a unipotent element in $G$.
 \end{lemma}

 \begin{proof}
 Suppose $C_{\overline{G}}(\overline{s}) > C_{\overline
   G}(\overline{s})^\circ=\overline{T}$.
The result follows from Lemma~\ref{lem:evencent} when
$|C_{\overline G}(\overline{s})|$ is even. Thus we may assume
$|C_{\overline G}(\overline{s})|$ is odd. In particular,
$[C_{\overline G}(\overline{s}):C_{\overline  G}(\overline{s})^\circ]$
is odd and also divides  the  order $k$ of the  Fundamental group  as in
 Table~\ref{tab:k}.  This implies we are
 either in case $A_\ell$, ${}^2A_\ell$, $E_6$ or  ${}^2E_6.$

If $|\overline{s}|$ is not prime, then we can choose $m\in \N$ such
that $|\overline{s}^m|$ is prime. If $\overline{s}^m$ is not regular,
then $\overline{s}^m$, and thus also $\overline{s}$, is equivalent to
a unipotent element by  Lemma~\ref{lem:semi-non-reg}. Thus we may
assume $\overline{s}^m$ is regular. Then $T$ is the unique torus containing
$\overline{s}^m$. Moreover, $\overline{s}^m \in \overline{T} =
C_{\overline   G}(\overline{s})^\circ <
C_{\overline{G}}(\overline{s}) \le
C_{\overline{G}}(\overline{s}^m).$ Thus, by
Lemma~\ref{lem:centconnected},
$|\overline{s}^m|$ is prime divisor of $k$.
Hence, replacing $\overline{s}$ by  $\overline{s}^m$ if necessary,
we may  now assume that $|\overline{s}|$ is an  odd prime dividing
$k$.

Suppose first we are in
case $A_\ell$ or ${}^2A_\ell$.  In  this case, a semisimple element in
$\GL^\epsilon(n,q)$  is  regular if  and  only  if its  characteristic
polynomial is equal to its minimal polynomial.
 As  $k$ divides  $q-\epsilon$,  a
regular  semisimple element  of order  dividing  $k$ is  similar to  a
diagonal  matrix   with  pairwise   distinct  entries   in  underlying
field.  Thus   the  unique  torus  inside   $\GL(n,q)$  or  $\GU(n,q)$
containing this element is  isomorphic to $\Z_{q -\epsilon}^n$.  Inside
$\overline{G}$ it follows that $|\overline{T}|$  is even for $n\ge 3$,
contrary to  our assumption that $|C_{\overline  G}(\overline{s})|$ is
odd. Hence this case cannot arise.

Now consider the case that $\overline{G}=E_6(q)$  or $\overline{G} =
{}^2E_6(q)$. By \cite[Example~26.11]{MalleTesterman}
$C_{G}(\overline{s})$ is one of the centralizers in
\cite[Table~26.1]{MalleTesterman} and divisible by $q$. In particular, $s$ is
equivalent to a unipotent element.
 \end{proof}

 \begin{lemma}\label{lem:semiregandnonreg}
Let $G$ be as in Section~$\ref{sec:ourgroups}$ and
let $s$ be a regular semisimple element in $G$ such that
 $C_G(s)^\circ$ a maximal torus which is not a $CC$-group.
Moreover, suppose
  $C_{\overline G}(\overline{s}) =  C_{\overline G}({\overline s})^\circ=\overline{T}$.
Then  $s$ is
 equivalent to a unipotent element.
 \end{lemma}

 \begin{proof}
 For any $t\in T$ we have $Z(G) \le C_G(t)$ and,
 since $T$ is abelian,
 $T \le C_G(t)$.  Moreover, $\overline{C_G(t)} \le
 C_{\overline{G}}(\overline{t})$. For $t=s$ this implies $\overline{T} \le
 \overline{C_G(s)} \le
 C_{\overline{G}}(\overline{s}) = \overline{T}$ and hence  $C_G(s) = T.Z(G)$
 and  $C_G(s) \le C_G(t)$ for all $t\in T$, implying
 $s\sim t.$

 As $\overline{T}$ is not a $CC$-group,
 there is one  element $t\in T$ such that $\overline{T} < C_{\overline
   G}(\overline{t})$.
 If  $t$ is  non-regular semisimple, then
 by Lemma~\ref{lem:semi-non-reg}    there is a
 unipotent element $u$  equivalent to $t$ while for
 $t$ regular semisimple such a $u$ exists by
 Lemma~\ref{lem:aux}.
 Therefore, $s$ and $u$ are also
 equivalent.
 \end{proof}

 \begin{lemma}\label{lem:semi2}
Let $G$ be as in Section~$\ref{sec:ourgroups}$.
 Then every non-central  semisimple element $s\in G$
 is either equivalent to a unipotent element of $G$ or  $|s^G|$ is an
 isolated vertex of $\D(G)$.
 \end{lemma}

\begin{proof}
 As $\overline{s}\in \overline{G}$ is  semisimple, there exists a
 maximal torus $\overline{T}$ of $\overline{G}$
such that $\overline{s}\in \overline{T}$. We consider the following
cases:

{\bf Case 1:}\quad $|C_{\overline{G}}(\overline{s})|$ is even. Then
$s$ is equivalent to a unipotent element by Lemma~\ref{lem:evencent}.
 From now on we assume that  $|C_{\overline{G}}(\overline{s})|$ is
 odd.

 {\bf Case 2:}\quad   $\overline{T}$ is a $CC$-subgroup.
Then   $\gcd(|Z(G)|, |\overline{T}|)=1$ by Lemma~\ref{lem:nonisotori}
as in particular $|\overline{T}|$ is odd. Therefore
$s$ is related to an isolated vertex by Lemma~\ref{lem:isotori}.
From now on we assume that $\overline{T}$ is not a $CC$-group.

{\bf Case 3:}  $s$ is regular semisimple, i.e.
$C_{\overline G}(\overline{s})^\circ=\overline{TZ(G)}.$
 If $C_{\overline G}({\overline s}) > C_{\overline
   G}(\overline{s})^\circ=\overline{TZ(G)}$ then $s$ is
 equivalent to a unipotent element in $G$ by Lemma~\ref{lem:aux}.
 If $C_{\overline G}({\overline s}) = C_{\overline
   G}(\overline{s})^\circ=\overline{TZ(G)}$ then $s$ is
 equivalent to a unipotent element in $G$ by Lemma~\ref{lem:semiregandnonreg}.

{\bf Case 4:} $s$ is a  non-regular semisimple element.
By Lemma~\ref{lem:semi-non-reg} $s$ is equivalent to a unipotent element.
\end{proof}

\section*{acknowledgements}
We  thank  Frank  L\"ubeck  for  providing   lists  of  the  generic
conjugacy class sizes  of the exceptional groups  $E_7(q)$ computed in
{\sf  GAP} and for helpful discussions.
The  third  author
acknowledges  the  support  of  the
Australian Research Council Discovery Project DP140100416.

\bibliographystyle{abbrvnat}

\begin{thebibliography}{16}
\providecommand{\natexlab}[1]{#1}
\providecommand{\url}[1]{\texttt{#1}}
\expandafter\ifx\csname urlstyle\endcsname\relax
\providecommand{\doi}[1]{doi: #1}\else
\providecommand{\doi}{doi: \begingroup \urlstyle{rm}\Url}\fi


\bibitem{AI1}
A.~Abdolghafourian and M.~A.~Iranmanesh, \newblock{Divisibility graph for
symmetric and alternating groups. }
\emph{Comm. Algebra} {\bf 43} (7):
2852-2862, 2015.

\bibitem{AI2}
A.~Abdolghafourian and M.~A.~Iranmanesh, On the number of connected
components of divisibility graph for certain simple groups.
\emph{Transactions on Combinatorics}, {\bf 5} (2),33-40,  2016.

\bibitem{AIrocky}
{A.~Abdolghafourian} and {M.~A.~Iranmanesh}, 
\newblock {On divisibility graphs for finite sets of positive integers.}
\emph{Rocky Mountain J. Math.}, (to appear).

\bibitem{Babaietal}
L.~Babai, P.~P{\'a}lfy and J.~Saxl,
 {On the number of {$p$}-regular elements in finite simple
             groups},
\emph{LMS J. Comput. Math.},
%  FJOURNAL = {LMS Journal of Computation and Mathematics},
 {\bf 12},
{(2009)}, {82--119}.


\bibitem{BHM}
 {Bertram, Edward A. and Herzog, Marcel and Mann, Avinoam},
{On a graph related to conjugacy classes of groups},
\emph{Bull. London Math. Soc.},
 {\bf 22}(6),
{(1990)},
 {569--575}.

\bibitem{BrauerFowler55}
R.~Brauer and K.~A.~Fowler,
\newblock{On groups of even order},
\emph{Annals of Mathematics} \textbf{62} (3), (1955), {565--583}.

\bibitem{BDIP}
D.~Bubboloni, S.~Dolfi~, M.~A.~Iranmanesh and {C.~E.~Praeger},
{On bipartite divisor graphs for group conjugacy class sizes}. \newblock {\em J. Pure Appl. Algebra}, \textbf{213}(9) (2009)1722--1734.



\bibitem{CaminaCamina}
A.~R.~Camina and R.~D.~Camina,
{The influence of conjugacy class sizes on the structure of
              finite groups: a survey},
\emph{Asian-Eur. J. Math.} {\bf 4}(4),  (2011),
 {559--588}.

\bibitem{Carter}
R.~W.~Carter, \emph{Finite groups of Lie type}. Wiley Classics
Library.  John Wiley \& Sons Ltd., Chichester, (1993).
%ISBN 0-471-94109-3.
Conjugacy classes and complex characters,
Reprint of the 1985 original, A Wiley-Interscience Publication.

\bibitem{CarterPLMS}
R.~W.~Carter, \emph{Centralizers of semisimple elements in the finite
  classical groups}. \emph{Proc. London Math. Soc.} {\bf 3} 42 (1981) 1--41.

\bibitem{Chang}
B.~Chang, The conjugate classes of Chevalley groups of type ($G_2$). \emph{J. Algebra} \textbf{9} (2) (1968) 190--211.

\bibitem{Der1983}
D.~I.~Deriziotis, The Centralizers of Semisimple Elements of the Chevalley Groups $ E_7 $ and $ E_8$. \emph{Tokyo J. Math.} \textbf{6} (1) (1983) 191--216.

\bibitem{DigneMichel}
F.~Digne and J.~ Michel,
\emph{Representations of finite groups of Lie type},
 {London Mathematical Society Student Texts},
   {\bf 21},
{Cambridge University Press, Cambridge},
{(1991)}.

\bibitem{DM}
D.~I.~Deriziotis and G.~O.~Michler, Character table and blocks of finite simple triality groups ${}^3D_4 (q)$. \emph{Trans. Amer. Math. Soc.} \textbf{303} (1) (1987) 39--70.		

\bibitem{Dolfi}
S.~Dolfi, Arithmatical conditions on the length of the conjugacy classes of a finite group. \emph{J. Algebra} \textbf{174} (1995) 753--771.

\bibitem{Eno}
H.~Enomoto, The characters of the finite Chevalley group $G_2(q), q= 3^f$. \emph{Jpn. J. Math.} \textbf{2}  (2) (1976) 191--248.

\bibitem{FA}
{E.~Fisman } and {Z.~Arad},
{A proof of Szep's conjecture on nonsimplicity of certain finite
  groups}.  {\em J. Algebra}, \textbf{108}(2) (1987)  340--354.

\bibitem{FJ93}
P.~ Fleischmann  and I.~Janiszczak. The semisimple conjugacy
classes of finite groups of Lie type $E_6$ and $E_7$.
\emph{Comm. Algebra} 21.1 (1993), 93-161.

\bibitem{FJ94}
P.~Fleischmann and I.~Janiszczak, The semisimple conjugacy classes and the
generic class number of the finite simple groups of lie type $E_8$.
\emph{Comm. Algebra}  \textbf{22} (6) (1994),  2221--2303.

\bibitem{FNP}
J.~Fulman, P.~M.~Neumann and C.~E.~Praeger.
{A generating function approach to the enumeration of matrices
              in classical groups over finite fields}.
\emph{Mem. Amer. Math. Soc.},
\textbf {176} (830),  (2005).

\bibitem{GAP}
{The \GAP~Group},
{\GAP--Groups, Algorithms, and Programming},
{\rm{4.7.8}},
{2015},
{\scalebox{0.9}{\tt http://www.gap-system.org}}.
%  note={{\fontsize{8}{8}\tt http://www.gap-system.org}},

\bibitem{GoodwinRoehrle}
S.~M.~Goodwin and G.~Roehrle.
{On conjugacy of unipotent elements in finite groups of Lie type},
\newblock\emph{J. of Group Theory} \textbf{12} (2) (2009), 235--245.

\bibitem{IranmaneshJafarzadeh}
A.~Iranmanesh and A.~Jafarzadeh,
{On the commuting graph associated with the symmetric and
              alternating groups},
 \emph{J. Algebra Appl.},
\textbf{7}(1) (2008)   {129--146}.

\bibitem{IP}
{M.~A.~Iranmanesh} and {C.~E.~Praeger}, {Bipartite divisor graphs for integer subsets}.  {\em Graphs Combin.}, \textbf{26}(1) (2010)  95--105.
\bibitem{kazarin}
{L.~S.~Kazarin}, {On groups with isolated conjugacy classes}.
{\em Izv. Vyssh. Uchebn. Zaved. Mat.}, \textbf{25}(7) (1981) 40--45.

\bibitem{Kondratev}
A.~S.~Kondrat'ev,
{Prime graph components of finite simple groups},
\emph{Math. USSR Sbornik},
\textbf{67}(1) (1990) {235--247}.

\bibitem{Lewis}
{M.~L.~Lewis}, {An overview of graphs associated with character degrees and conjugacy class sizes in finite groups}.
{\em Rocky Mountain J. Math.}, \textbf{38}(1) (2008) 175--211.

\bibitem{LS}
M.~W.~Liebeck and G.~M.~Seitz, \emph{Unipotent and nilpotent classes in simple algebraic groups and Lie
algebras}, Mathematical Surveys and Monographs, 180. American
Mathematical Society, Providence, RI, (2012).

\bibitem{lucido}
M.~S.~Lucido, {Prime graph components of finite almost simple groups}.
\emph{Rend. Sem. Mat. Univ. Padova},
{\bf 102}, (1999), {1--22}.

\bibitem{ModularAtlas}
C.~Jansen, K.~Lux, R.~Parker and R.~Wilson.
\newblock \emph{The Atlas of Brauer Characters},
\newblock{Oxford University Press}, Oxford (1995).

\bibitem{MorganParker}
G.L.~Morgan and C.W.~Parker, {The diameter of the commuting graph of a finite group with trivial centre},\emph{J. Algebra},
\textbf {393}, (2013) {41--59}.

\bibitem{Mizuno}
K.~Mizuno, The conjugate classes of Chevalley groups of type $E_6$, \emph{J. Fac. Sci. Univ. Tokyo} \textbf{24} (1977) 525--563.

\bibitem{MalleTesterman}
 G.~Malle and D.~Testerman,
\newblock \emph{Linear algebraic groups and finite groups of {L}ie type},
\newblock {Cambridge Studies in Advanced Mathematics},
{\bf 133}, \newblock {Cambridge University Press, Cambridge} (2011).

\bibitem{Shoji}
T.~Shoji and N.~Iwahori, The conjugacy classes of Chevalley groups of
type $(F_4) $ over finite fields of characteristic $ p\neq
2$. \emph{J. Fac. Sci. Univ. Tokyo} \textbf{21} (1) (1974) 1--17.

\bibitem{ss}
W.~A.~Simpson and J.~Sutherland Frame, The character tables for
SU $(3, q^2)$, PSL $(3, q)$, PSU $(3, q^2)$.
\emph{Canad. J. Math} \textbf{25} (3) (1973)  486--494.

\bibitem{SpringerSteinberg}
T.~A.~ Springer and R.~Steinberg,
\emph{Conjugacy classes},
{Seminar on {A}lgebraic {G}roups and {R}elated {F}inite {G}roups
({T}he {I}nstitute for {A}dvanced {S}tudy, {P}rinceton, {N}.{J}., 1968/69)},
 {Lecture Notes in Mathematics, Vol. 131},
{Springer, Berlin}, (1970).

\bibitem{Williams}
J.~S.~Williams, Prime graph components of finite groups. \emph{J. Algebra} \textbf{69} (2) (1981) 487--513.

\end{thebibliography}
\def\cprime{$'$}

\bigskip
\noindent
Adeleh~Abdolghafourian\\
Department of Mathematics,\\
Yazd University, Yazd, 89195-741,\\
Iran\\
{\tt iranmanesh@yazd.ac.ir}

\bigskip
\noindent
Mohammad~A.~Iranmanesh\\
Department of Mathematics,\\
Yazd University, Yazd, 89195-741,\\
Iran\\
{\tt iranmanesh@yazd.ac.ir}

\bigskip
\noindent
Alice~C.~Niemeyer\\
Lehrstuhl B f\"ur Mathematik,\\
Pontdriesch 10-16, RWTH Aachen University, 52062 Aachen,\\
Germany\\
{\tt alice.niemeyer@mathb.rwth-aachen.de}

\end{document}